\documentclass[10pt]{amsart}
\usepackage{enumerate,amsmath,amssymb,latexsym,
amsfonts, amsthm, amscd, mathabx}

\theoremstyle{plain}

\newtheorem{theorem}{Theorem}

\numberwithin{equation}{section}

\addtocounter{section}{-1}

\begin{document}

\title {Bessel function asymptotics: a relation due to Lommel}

\date{}

\author[P.L. Robinson]{P.L. Robinson}

\address{Department of Mathematics \\ University of Florida \\ Gainesville FL 32611  USA }

\email[]{paulr@ufl.edu}

 \keywords{}

\begin{abstract}

We show that the Bessel function asymptotic relation $J_{\nu}^2(z) + J_{\nu + 1}^2 (z) \sim 2/(\pi z)$ of Lommel is valid when $z$ is real but can fail otherwise. 

\end{abstract}

\maketitle

\medbreak 

Our aim here is to elucidate and elaborate upon an asymptotic relation involving the squares of two Bessel functions that have orders differing by unity. Attention is 
drawn to this relation by Watson on page 200 of his magnificent `{\it Treatise on the theory of Bessel functions}' [1] as follows. 

\begin{quote} 
`The reader should notice that 
$$J_{\nu}^2(z) + J_{\nu + 1}^2 (z) \sim 2/(\pi z),$$
a formula given by Lommel, {\it Studien}, p. 67.'
\end{quote} 
\medbreak
\noindent	

\noindent 
In the study [2] of Lommel to which Watson refers, the relation actually appears on page 65 and is presented as an equality, with a verbal comment on its asymptotic nature. The symbol $\sim$ can signify various kinds of asymptotic relation, and care must be exercised in its use. Its use in the relation of Lommel is a case in point: as we shall see below, when $\sim$ is given its most familiar interpretation, the Lommel relation is not true in the general cut-plane context of [1].  

\medbreak 

We shall begin by recalling the definition and asymptotic development of a Bessel function (of the first kind). The Bessel function $J_{\nu}$ of complex order $\nu$ is defined by 
$$J_{\nu} (z) = (\tfrac{1}{2} z)^{\nu} \sum_{m = 0}^{\infty} (-)^m \frac{(\tfrac{1}{2} z)^{2 m}}{\Gamma(\nu + m + 1) \: m!}$$
where $z$ is any point in the complex plane cut along the negative real half-line $(- \infty, 0]$ and where the power $z^{\nu}$ is assigned its principal value (at least initially). The asymptotic expansion of $J_{\nu}(z)$ for $|z|$ large is as follows: when the square-root has its principal value,
$$(\tfrac{1}{2} \pi z )^{1/2} J_{\nu} (z) = \cos(z - \tfrac{1}{2} \nu \pi - \tfrac{1}{4} \pi) C(z) - \sin(z - \tfrac{1}{2} \nu \pi - \tfrac{1}{4} \pi) S(z)$$
where 
$$C(z) \sim \sum_{m = 0}^{\infty} (-)^m \frac{(\nu, 2m)}{(2 z)^{2 m}}$$
and 
$$S(z) \sim \sum_{m = 0}^{\infty} (-)^m \frac{(\nu, 2m + 1)}{(2 z)^{2 m + 1}}$$
and where 
$$(\nu, m) = \frac{(4 \nu^2 - 1^2) (4 \nu^2 - 3^2) \cdots (4 \nu^2 - (2 m - 1)^2)}{2^{2 m} m!}$$
is the customary Hankel symbol. In these asymptotic relations, the symbol $\sim$ is used in one of the senses customary for asymptotic power series: when $p$ is any positive integer, the difference  
$$ C(z) - \sum_{m = 0}^{p - 1} (-)^m \frac{(\nu, 2m)}{(2 z)^{2 m}}$$
is $O(z^{- 2 p})$ in the sense that 
$$z^{2 p} \Big( C(z) - \sum_{m = 0}^{p - 1} (-)^m \frac{(\nu, 2m)}{(2 z)^{2 m}}\Big)$$ 
remains bounded as $|z| \rightarrow \infty$; similarly for $S$. For a detailed account of these asymptotics, see pages 196-199 of [1]. 

\medbreak 

In particular, we have
$$(\tfrac{1}{2} \pi z )^{1/2} J_{\nu} (z) = c(z) \{ 1 + C_0 (z) \} - s(z)  S_0 (z) $$
where 
$$c(z) := \cos (z - \tfrac{1}{2} \nu \pi - \tfrac{1}{4} \pi)$$
$$s(z) := \sin( z - \tfrac{1}{2} \nu \pi - \tfrac{1}{4} \pi) $$
\medbreak
\noindent
and where the terms $C_0 (z)$ and $S_0 (z)$ are $O( z^{-1})$ or better. 

\medbreak 

Increase the order by unity, from $\nu$ to $\nu + 1$: on account of the identities
$$\cos(z - \tfrac{1}{2} (\nu + 1) \pi - \tfrac{1}{4} \pi) = c( z - \tfrac{1}{2} \pi) = s(z)$$
$$\sin( z - \tfrac{1}{2} (\nu + 1)\pi - \tfrac{1}{4} \pi) = s( z - \tfrac{1}{2} \pi) = - c(z)$$
we have 
$$(\tfrac{1}{2} \pi z )^{1/2} J_{\nu + 1} (z) = s(z) \{ 1 + C_1 (z) \} + c(z) S_1 (z)$$
\medbreak
\noindent
where $C_1(z)$ and $S_1(z)$ are $O(z^{-1})$ likewise. 

\medbreak 

Square and add: as $c(z)^2 + s(z)^2 = 1$ we deduce that 
$$\tfrac{1}{2} \pi z  \{ J_{\nu}^2(z) + J_{\nu + 1}^2 (z) \} - 1 = A(z) c(z)^2 + 2 H(z) c(z) s(z) + B(z) s(z)^2$$
where 
$$A(z) = C_0(z)^2 + 2 C_0(z) + S_1(z)^2$$
$$H(z) = C_1(z) S_1(z) - C_0(z) S_0(z) + S_1(z) - S_0(z)$$
$$B(z) = C_1(z)^2 + 2 C_1(z) + S_0(z)^2.$$

\medbreak 

We are now prepared to discuss the relation of Lommel. 

\medbreak 

\begin{theorem} \label{real}
For any complex order $\nu$, the difference 
$$\tfrac{1}{2} \pi t  [ J_{\nu}^2(t) + J_{\nu + 1}^2 (t) ] - 1 $$
is $O(t^{-1})$ as the positive real $t$ tends to infinity. 
\end{theorem} 

\begin{proof} 
Write $ - \tfrac{1}{2} \nu \pi - \tfrac{1}{4} \pi = a + i b$ with $a$ and $b$ real and fixed. It follows that 
$$c(t) = \cos ( t + a + i b) = \cos (t + a) \cosh (b) - i \sin (t + a) \sinh (b)$$
and 
$$s(t) = \sin (t + a + i b) = \sin (t + a) \cosh (b) + i \cos (t + a) \sinh (b)$$
are bounded as $t$ varies over the reals. The theorem now follows at once from the formulae displayed immediately prior to its statement. 
\end{proof} 

\medbreak 

Consequently, the Lommel relation 
$$J_{\nu}^2(z) + J_{\nu + 1}^2 (z) \sim 2/(\pi z)$$
is valid for each complex order $\nu$ when $z$ tends to infinity through (positive) {\it real} values; here, $\sim$ has the familiar meaning according to which the two sides to the relation have ratio approaching unity. We remark that at the outset of [2] Lommel declares an interest in {\it real} values of $\nu$ but does not restrict $z$ to be real; however, his relation can fail without some such restriction, as we now proceed to demonstrate. 

\medbreak 

When we allow $z$ to pass to infinity in an other-than-real direction, the coefficients $c(z)$ and $s(z)$ in the asymptotic formulae displayed prior to Theorem \ref{real} can grow exponentially, thereby counteracting the power decay of $A(z), H(z), B(z)$. A single example will suffice for the demonstration: we take $\nu = \tfrac{1}{2}$ (real) and let $z$ run to infinity up the imaginary axis. 

\medbreak 

Recall that if the complex number $z$ lies in the cut plane then 
$$(\tfrac{1}{2} \pi z)^{1/2}  J_{1/2} (z) = \sin (z)$$
and 
$$(\tfrac{1}{2} \pi z)^{1/2}  J_{3/2} (z) = \frac{\sin (z)}{z} - \cos (z)$$
whence 
$$(\tfrac{1}{2} \pi z) [ J_{1/2}^2(z) + J_{3/2}^2 (z) ] = \frac{\sin^2 (z)}{z^2} - 2 \frac{\sin(z) \cos(z)}{z} + 1.$$
In this case, taking $z = t$ to be real yields a concrete version of Theorem \ref{real}, whereas taking $z = i t$ to be pure imaginary yields the following. 

\medbreak 

\begin{theorem} \label{imag}
$$\lim_{t \rightarrow \infty} \frac{2 t}{e^{2 t}}\Big\{(\tfrac{1}{2} \pi i t)[ J_{1/2}^2(i t) + J_{3/2}^2 (i t) ] - 1 \Big\} = - 1.$$
\end{theorem} 

\begin{proof} 
From the formula displayed immediately before the theorem, it follows that if $t$ is a positive real number then 
$$(\tfrac{1}{2} \pi i t) [ J_{1/2}^2(i t) + J_{3/2}^2 (i t) ] - 1 = \frac{\sinh^2 (t)}{t^2} - \frac{\sinh(2 t)}{t}$$
so that  
$$\frac{2 t}{e^{2 t}}\Big\{(\tfrac{1}{2} \pi i t)[ J_{1/2}^2(i t) + J_{3/2}^2 (i t) ] - 1 \Big\}  = \frac{(1 - e^{- 2 t})^2}{2 t} - 1 + e^{- 4 t}$$
\medbreak
\noindent
whence passage to the limit concludes the proof. 

\end{proof} 

\medbreak 

Thus 
$$\Big\{ \frac{J_{1/2}^2(i t) + J_{3/2}^2 (i t)}{\tfrac{2}{\pi i t}} - 1 \Big\} \sim - \frac{e^{2 t}}{2 t}$$
\medbreak
\noindent
in the sense that the two sides have ratio tending to unity as $t$ tends to infinity; so the Lommel relation fails in quite spectacular fashion. 

\medbreak 

In fact, for the Bessel functions of Theorem \ref{imag} the Lommel relation fails whenever $z$ tends to infinity along {\it any} non-real ray through the origin. Let $z = t e^{i \theta}$ with $t > 0$ variable and $0 < \theta < \pi$ fixed: an argument along the lines of that for Theorem \ref{imag} shows that 
$$(\tfrac{1}{2} \pi z)  [ J_{1/2}^2(z) + J_{3/2}^2 (z) ] - 1 = \frac{e^{- 2 i z}}{2 z} F(t, \theta) = \frac{e^{2 t \sin \theta - 2 i t \cos \theta}}{2t e^{i \theta}} F(t, \theta)$$
where $F(t, \theta) \rightarrow -i$ as $t$ tends to infinity, so the Lommel relation is violated as $z \rightarrow \infty$ along the ray $\arg z = \theta$ with $0 < \theta < \pi$ since there $\sin \theta > 0$; extracting instead the factor $e^{2 i z}/{2 z}$ shows that the Lommel relation is likewise violated along rays in the lower half-plane. In short, for $\nu = \tfrac{1}{2}$ and along rays through the origin, the Lommel relation is valid {\it only} when the argument $z$ is real.   
\medbreak 

Naturally, in Theorem \ref{imag} we choose $\nu$ to be half an odd integer because $J_{\nu}(z)$ is then an elementary function of $z$ and the argument is  elementary as a result. Of course, we do not find a counterexample to the Lommel relation by choosing $\nu = - \tfrac{1}{2}$, for then the Lommel relation certainly holds in its original form, with true equality: 
$$J_{ - 1/2}^2(z) + J_{1/2}^2 (z) = \frac{2}{\pi z}.$$

\medbreak 

Finally, there is one (largely figurative) sense of the symbol $\sim$ according to which the Lommel relation {\it does} hold for any complex $\nu$ and throughout the cut plane. It holds at the level of leading terms for the Bessel functions: the leading terms in the asymptotic expansions of $J_{\nu}(z)$ and $J_{\nu + 1}(z)$ are $(\tfrac{2}{\pi z})^{1/2} c(z)$ and $(\tfrac{2}{\pi z})^{1/2} s(z)$ respectively; these square and add to $\frac{2}{\pi z}$ precisely. Indeed, Lommel sets up his asymptotic relation in just this way on page 65 of [2].

\medbreak

\bigbreak 

\begin{center} 
{\small R}{\footnotesize EFERENCES}
\end{center} 
\medbreak 

[1] G.N. Watson, {\it A Treatise on the theory of Bessel functions}, Cambridge University Press, Second Edition (1944). 

\medbreak 

[2] E.C.J. von Lommel, {\it Studien \"uber die Bessel'schen Functionen}, Teubner, Leipzig (1868). 

\medbreak

\end{document}